\documentclass[10pt, bezier]{article}
\usepackage[toc,page]{appendix}
\usepackage{amsthm,amssymb}
\usepackage{amsmath}
\usepackage{graphicx,tikz,pgf}
\usepackage{color}
\usepackage{xcolor}
\usepackage{caption}
\usepackage[T2A,T1]{fontenc}
\usepackage[utf8]{inputenc}
\usepackage[russian,english]{babel}
\usepackage{subcaption}
\usepackage{fancyhdr,cancel,enumerate,array,booktabs,setspace}
\usetikzlibrary{patterns,arrows,decorations.pathreplacing}
\usepackage{mathrsfs}
\usetikzlibrary{arrows}
\usepackage{epic}
\usepackage{eepic}
\usepackage{shadow}
\usepackage{url}
\definecolor{xdxdff}{rgb}{0,0,0}
\definecolor{qqqqff}{rgb}{0,0,0}
\definecolor{uuuuuu}{rgb}{0,0,0}
\definecolor{ududff}{rgb}{0,0,0}
\date{ }
\theoremstyle{definition}

\newtheorem{definition}{Definition}

\newtheorem{theorem}[definition]{Theorem}
\newtheorem{lemma}[definition]{Lemma}

\newtheorem{corollary}[definition]{Corollary}

\newcommand{\emat}[2][p]{\mathcal{E}_{#1}(#2)}
\newcommand{\se}[2][]{\mathcal{E}_{#1}(\mathcal{S}(#2))}
\newcommand{\sm}[1]{\mathcal{S}(#1)}
\newcommand{\nop}[1]{s(#1)}
\newcommand*{\QEDB}{\hfill\ensuremath{\square}}%

\title{\bf Proof of a Conjecture on the Seidel \\ Energy of Graphs}
\author{\bf \small{S. Akbari$^{{\rm a}}$\footnote{E-mail addresses:
			s$\_$akbari@sharif.edu (S. Akbari), eynollahzadehsamadi@gmail.com (M. Einollahzadeh), karkhaneei@gmail.com (M. M. Karkhaneei) mohammadali.nematollahi69@student.sharif.edu (M.A. Nematollahi).}, M. Einollahzadeh, M. M. Karkhaneei$^{{\rm a}}$ and M. A. Nematollahi$^{{\rm a}}$}
\\[2mm]
${\rm ^{a}}$\small Department of Mathematical Sciences, Sharif University
of Technology, Tehran, Iran}
\begin{document}

\maketitle
\begin{abstract}
  Let $G$ be a graph with the vertex set $ \lbrace v_1,\ldots,v_n \rbrace$.
  The Seidel matrix of $G$ is an $n\times n$ matrix whose diagonal entries are zero, $ij$-th entry is $-1$ if
  $ v_{i} $ and $ v_{j} $ are adjacent and otherwise is
  $ 1 $. The Seidel energy of $G$ is defined to be the sum
  of absolute values of all eigenvalues of the Seidel
  matrix of $G$. Haemers conjectured that the Seidel
  energy of any graph of order $n$ is at least $2n-2$ and, up to Seidel equivalence, the equality holds for $ K_{n} $. 
  We establish the validity of Haemers' Conjecture in general. 
\end{abstract}
 \vskip 3mm

\section{Introduction and Terminology}
Throughout this paper all graphs we consider are simple and finite. For a graph $ G $, we denote the set of vertices and edges of $ G $ by $ V(G) $ and $ E(G) $, respectively. The complement of $ G $ denoted by $ \overline{G} $ and the complete graph of order $ n $ is denoted by $ K_{n} $. In this paper, for $ v \in V(G) $, $ N_{G}(v) $ and $ N_{G}[v] $ denote the open neighborhood and the close neighborhood of $ v $ in $ G $, respectively.

For every Hermitian matrix $A$ and any real number $ p > 0 $, the \textit{$p$-energy} of $A$, $\emat{A}$,  is defined to be sum of the $p$-th power of absolute values of the eigenvalues of $A$. The well-known concept of energy of a graph $ G $ denoted by $ \mathcal{E}(G) $ is $\emat[1]{A}$, where $ A $ is the adjacency matrix of $ G $. 

Let $ G $ be a graph and $ V(G)=\lbrace v_{1}, \ldots , v_{n} \rbrace $. The \textit{Seidel matrix} of $ G $, denoted by $ \sm{G} $, is an $ n \times n $ matrix whose diagonal entries are zero, $ij$-th entry is $-1$ if $ v_{i} $ and $ v_{j} $ are adjacent and otherwise is $ 1 $ (It is noteworthy that at first, van Lint and Seidel introduced the concept of Seidel matrix for the study of equiangular lines in \cite{lint-seidel}). The \textit{p-Seidel energy} of $ G $ is defined to be    $\se[p]{G}$. By the \textit{Seidel energy} of $G$, we mean $1$-Seidel energy of $G$ and denote by $\se{G}$. The \textit{Seidel switching} of $ G $ is defined as follows: Partition $ V(G) $ into two subsets $ V_{1} $ and $ V_{2} $, delete the edges between $ V_{1} $ and $ V_{2} $ and join all vertices $ v_{1} \in V_{1} $ and $ v_{2} \in V_{2} $ which are not adjacent. Therefore, if we call the new graph by $ G^{\prime} $, then we have $\sm{G^{\prime}}= D \sm{G} D $, where $ D $ is a diagonal matrix with entries 1 (resp. $ -1 $) corresponding to the vertices of $ V_{1} $ (resp. $ V_{2} $) (\cite{Ham}). Hence, $ \sm{G} $ and $ \sm{G^{\prime}} $ are similar and they have the same Seidel energy. Note, that if one of the $ V_{1} $ or $ V_{2} $ is empty, then $ G $ remains unchanged. Two graphs $ G_{1} $ and $ G_{2} $ are called \textit{S-equivalent} (resp. \textit{SC-equivalent}) if $ G_{2} $ is obtained from $ G_{1} $ (resp. $G_1$ or $ \overline{G_{1}}) $ by a Seidel switching. Note that in either cases, $ \sm{G_{2}} $ is similar to $ \sm{G_{1}} $ or $ -\sm{G_{1}} $, hence $ \se{G_{1}}=\se{G_{2}} $.
	
	For every square matrix $ A $ with eigenvalues $ \lambda_{1}, \ldots , \lambda_{n} $, by $ S_{k}(A) $ we mean $ S_{k}(\lambda_{1}, \ldots , \lambda_{n}) $, where $ S_{k}(x_{1}, \ldots, x_{n}) $ is the \emph{$ k $-th elementary symmetric polynomial} in $ n $ variables; i.e. 
\begin{equation*}
S_{k}(x_{1}, \ldots, x_{n})=\sum \limits_{1 \leq i_{1} < \cdots < i_{k} \leq n} x_{i_{1}} \cdots x_{i_{k}}, \quad (S_{0}(x_{1}, \ldots, x_{n}):= 1).
\end{equation*}	
	Also, for every $ m \times n $ matrix $ R $ and the index sets $ I \subseteq \lbrace 1, \ldots , n \rbrace $ and $ J \subseteq \lbrace 1, \ldots , m \rbrace $, $ R_{I,J} $ is the submatrix of $ R $ obtained by the restriction of $ R $ to the rows $ I $ and the columns $ J $. By $ R^{\star} $, we mean conjugate transpose of $ R $.
	
	 Haemers in \cite{Ham} introduced the concept of Seidel energy of a graph and he proposed the following conjecture:
	 \vskip 2mm
	\noindent \textbf{Conjecture.} For every graph $ G $ of order $ n $, $ \se{G} \geq \se{K_{n}} = 2n-2 $.
	 \vskip 2mm
	  This conjecture was first investigated by Haemers for $ n \leq 10 $ and then was settled for $ n \leq 12 $ in \cite{gre}. Ghorbani in \cite{gho} proved Haemers' Conjecture for the graphs $ G $ of order $ n $ such that $ n-1 \leq \lvert \det (\sm{G}) \rvert $. Also, Oboudi in \cite{obu} proved Haemers' Conjecture for every $ k $-regular graph $ G $ of order $ n $ such that $ k \neq \frac{n -1}{2} $ and $ G $ has no eigenvalue in $ (-1,0) $. Here, we establish two following theorems which are the main results of the paper:
	\begin{theorem}\label{firstthm1}
		Let $ G $ be a graph of order $ n $. Then,
			for every real number  $ p \in (0,2) $
			\begin{equation*}
			\se[p]{G} > (n-1)^{p}+ (n-2).
			\end{equation*}
		\end{theorem}
	\begin{theorem}\label{firstthm2}
		For every graph $ G $ of order $ n $,
		$\se{G} \geq \se{K_n} =2n-2.$
			Moreover, if $ G $ and $ K_{n} $ are not SC-equivalent, then the inequality is strict.
	\end{theorem}
Hence, Theorem \ref{firstthm2} proves Haemers' Conjecture. Note that if one restricts the attention to the circulant graphs, then the nature of the Haemers' Conjecture resembles the sharp Littlewood Conjecture on the minimum of the $L^{1}$-norm of polynomials (on the unit circle in the complex plane) whose absolute values of coefficients are equal to $1$.\footnote{Thanks to  T. Tao's comment in   \url{https://mathoverflow.net/q/302424/53059}} For the special class of polynomials with $\pm 1$ coefficients, Kleme\v s proved the sharp Littlewood  Conjecture \cite{kle} and in the procedure of his proof, he used the following equality:
 \begin{equation}\label{expand}
\emat{A}:= \sum \limits_{i=1}^{n} \vert \lambda_{i} \vert^{p}= C_{\frac{p}{2}} \int_{0}^{\infty} \ln \left(\sum \limits_{k=0}^{n} S_{k}(A^{2}) t^{k} \right)t^{-\frac{p}{2}-1}dt, \quad \text{for every}~ p \in (0,2),
\end{equation}
where $ \lambda_{1}, \ldots , \lambda_{n} $ are the eigenvalues of the matrix $ A $.

The Equation \eqref{expand} comes from the Equation \eqref{kle1} which can be checked by a change of variable:
\begin{equation}\label{kle1}
\alpha^{p}= C_{p} \int_{0}^{\infty} \ln (1+ \alpha t) t^{-p-1} dt, \quad  C_{p}=\left(\int_{0}^{\infty} \ln (1+t)t^{-p-1}dt\right)^{-1},
\end{equation}
where $ p \in (0,1) $ and $ \alpha = r e^{i \theta} $ is a complex number which is not a negative real number, $  r > 0 $ and $ -\pi < \theta < \pi$.
Indeed, if for every natural number $ n $ and complex  numbers $ \alpha_{1}, \ldots , \alpha_{n} $ (no $ \alpha_{i}$ is negative real number) we define $ f(t)= \prod_{i=1}^{n} (1+t \alpha_{i}) $, then by Equation \eqref{kle1} we have
\begin{equation}\label{sum}
\sum_{i=1}^{n}\alpha_{i}^{p}= C_{p} \int_{0}^{\infty} \ln (f(t)) t^{-p-1} dt, \quad \text{for every}~ p \in (0,1).
\end{equation}
Now, one can expand $ f(t) $ as
$\sum \limits_{k=0}^{n} S_{k}(\alpha_{1}, \ldots , \alpha_{n})t^{k}$. Hence, if for every $ i,  $ $ \alpha_{i}=\lambda_{i}^{2} $, where $ \lambda_{1}, \ldots , \lambda_{n} $ are eigenvalues of $ A $, then Equation \eqref{expand} is obtained.

Throughout this paper, we consider the branch $ r^{p} e^{i p \theta} $ for $ \alpha^{p} $. We use Equation \eqref{expand} to establish a lower bound for  $S_{k}(A^2)$, where $A$ is the Seidel matrix of a graph, and then we  prove Haemers' Conjecture in general.

\section{Main Theorems}
In this section, we prove Theorems \ref{firstthm1} and \ref{firstthm2}. So Haemers' Conjecture holds. First, we need the following well-known identity, by the Cauchy-Binet Theorem (See \cite[p.776]{mo}).
 \begin{theorem}[Cauchy-Binet]\label{CB}
   If $B$ is an $n\times n$ matrix of the form $B = RR^{*}$ for some matrix $R$, then
\begin{equation*}
S_{k}(B)= \sum \limits_{I,J} (\det R_{I,J})^{2},
\end{equation*}	
where the summation is taken over all $ k $-subsets $ I, J \subset \{1, \ldots, n \}$.
\end{theorem}
\begin{lemma}\label{llb}
	Let $ G $ be a graph of order $ n $ and $ A=\sm{G} $. Then, for $ k=1,\ldots , n $, we have 
	\begin{equation*}
	S_{k}(A^{2}) \geq n(n-1)\binom{n-2}{k-1}.
	\end{equation*}
\end{lemma}
\begin{proof}
	By Theorem \ref{CB}, $ S_{k}(A^{2})= \sum \limits_{I,J}(\det (A_{I,J}))^{2} $, where the summation is taken over all $ k $-subsets $ I $ and $ J $. Now, we prove the following claim.
	\vskip 2mm
\noindent \textbf{Claim.} For every two subsets $ I $ and $ J $ such that $ \vert I \cap J \vert = k-1 $, $ \vert\det(A_{I,J}) \vert \geq 1  $.
	
\noindent To prove the claim, note that for every $ (i,j) $, if $ i \neq j $, then $ a_{i,j}= \pm 1 $ and otherwise $ a_{i,j}=0 $. So, if $ \vert I \cap J \vert = k-1 $, then after applying permutations on rows and columns of $ A_{I,J} $, modulo 2, $ A_{I,J} $ has the following form:
	\[\begin{pmatrix}
	1 & 1 & 1 & \cdots & 1 \\
	 1 & 0 & 1 & \cdots & 1 \\
	 1 & 1 & 0 & \cdots & 1 \\
	 \vdots & \vdots & & \ddots & \vdots \\
	 1 & 1 & 1 & \cdots & 0
	\end{pmatrix}.
	\]
	In the above matrix, subtract the first column from the other columns and get the following matrix:
		\[\begin{pmatrix}
	1 & 0 & 0 & \cdots & 0 \\
	1 & 1 & 0 & \cdots & 0 \\
	1 & 0 & 1 & \cdots & 0 \\
	\vdots & \vdots & & \ddots & \vdots \\
	1 & 0 & 0 & \cdots & 1
	\end{pmatrix},
	\]
	whose determinant is $ 1 $. So $ \det(A_{I,J}) $ is an odd number and the claim is proved. On the other hand, the number of pairs $ (I,J) $ with the above property is $ n (n-1) \binom{n-2}{k-1} $ and the proof is complete.
\end{proof}

\noindent\textbf{Proof of Theorem 1.}
Let $ A= \sm{G} $ and $ \lambda_{1}, \ldots , \lambda_{n} $ be its eigenvalues. By Equation \eqref{expand}, for every $ p \in (0,2) $, we have
	\begin{equation*}
	\mathcal{E}_{p}(A)= \sum \limits_{i=1}^{n} \vert \lambda_{i} \vert^{p}= C_{\frac{p}{2}} \int_{0}^{\infty} \ln (\sum \limits_{k=0}^{n} S_{k}(A^{2}) t^{k})t^{-\frac{p}{2}-1}dt.
	\end{equation*}
	Since $ \prod \limits_{i=1}^{n}(1+t \lambda_{i}^{2})= \sum \limits_{k=0}^{n} S_{k}(A^{2}) t^{k} $, by Lemma \ref{llb}, for every $ t \geq 0 $, we have
	\begin{equation}\label{eq1}
	 \prod \limits_{i=1}^{n}(1+t \lambda_{i}^{2})  = 1+ \sum \limits_{k=1}^{n} S_{k}(A^{2}) t^{k} 
	 \geq 1+ \sum \limits_{k=1}^{n-1} n(n-1) \binom{n-2}{k-1}t^{k} 
	 = 1+ n(n-1)t(1+t)^{n-2}.
	\end{equation}
	Note that for every  $t > 0$, we have 
	\begin{align*}
	(1+t)^{n-2}-1 & = t((1+t)+ \cdots + (1+t)^{n-3})\\
	& \leq t(n-3)(1+t)^{n-3}\\
	& < t(n-2)(1+t)^{n-3}.
	\end{align*}
	Hence, Inequality \eqref{eq1} can be written as follows:
	\begin{align*}
	\prod \limits_{i=1}^{n} (1+t \lambda_{i}^{2}) & \geq 1+ n(n-1)t(1+t)^{n-2}\\
	&  = 1 + (n-2)t(1+t)^{n-2}+(n^{2}-2n+2)t(1+t)^{n-2}\\
	& \geq 1 + (n-2)t(1+t)^{n-3}+(n^{2}-2n+2)t(1+t)^{n-2} \\
	& > (1+t)^{n-2}+ (n^{2}-2n+2)t(1+t)^{n-2}\\
	& = (1+t)^{n-2}(1+(n^{2}-2n+2)t).
	\end{align*}
	So, by Equation \eqref{expand}, for every $ 0 < p < 2 $,
	\begin{align*}
	\mathcal{E}_{p}(A) &= \sum \limits_{i=1}^{n} \vert \lambda_{i} \vert^{p}= C_{\frac{p}{2}} \int \limits_{0}^{\infty} \ln(\prod \limits_{i=1}^{n}(1+t \lambda_{i}^{2}))t^{-\frac{p}{2}-1}dt \\
	& > C_{\frac{p}{2}}((n-2) \int_{0}^{\infty} \ln (1+t) t^{-\frac{p}{2}-1}dt + \int_{0}^{\infty} \ln (1+(n^{2}-2n+2)t) t^{-\frac{p}{2}-1}dt) \\
	& = (n-2)+ (n^{2}-2n+2)^{\frac{p}{2}} > (n-2) + (n-1)^{p}.
	\end{align*}
	The proof is complete. \QEDB
\vskip 2mm
Now, if we apply Theorem \ref{firstthm1} for $ p=1 $, then we have the following corollary
\begin{corollary}
	For every graph $ G $ of order $ n $,  $ \se{G} > 2n-3 $.
\end{corollary}
 In order to prove Theorem \ref{firstthm2}, we strengthen the inequality of Lemma \ref{llb}. For this purpose, we pay attention to the $ 2 \times 2 $ submatrices of $ \sm{G} $ with an  odd number of $ -1 $, which have a key rule in the value of $ S_{k}(A^{2}) $ and hence  in $ \se{G} $. We state the following definition.
	\begin{definition}
		Let $ G $ be a graph. An ordered pair $ (X,Y) $ of disjoint subsets of $ V(G) $ with $ \vert X \vert = \vert Y \vert =2 $, is called an \textit{odd pair} if the number of edges with one endpoint in $ X $ and another in $ Y $ is odd. We denote the number of odd pairs in $ G $ by $ \nop{G} $.
	\end{definition}
Let $ (X,Y) $ be an odd pair. Consider the $ 2 \times 2 $ submatrix of $ A=\sm{G} $ whose rows and columns are corresponding to the vertices of $ X $ and $ Y $. This submatrix has the form
\[\begin{pmatrix}
	\pm 1 & \pm 1 \\
	\pm 1 & \pm 1
\end{pmatrix},
\]
which contains one or three $ -1 $. It is easily seen that the determinant of this matrix is $ \pm 2 $.
	\begin{lemma}\label{lskA2}
	Let $ G $ be a graph of order $ n $ and $ A=\sm{G} $. Then, for $ k = 1, \ldots, n-2 $, we have
		\begin{equation}\label{eq-t}
		S_{k}(A^{2}) \geq n(n-1) \binom{n-2}{k-1} + 4 \nop{G} \binom{n-4}{k-2}.
		\end{equation}
	\end{lemma}
\begin{proof}
	The term $ n(n-1) \binom{n-2}{k-1} $ on the right hand side of \eqref{eq-t} is deduced by the proof of Lemma \ref{llb}. Now, suppose that $ (X,Y) $ be an odd pair. If $ I \subset G $, $ \vert I \vert=k-2 $ and $ I \cap X= I \cap Y= \varnothing $, we claim that the absolute value of determinant of $ k \times k $ submatrix of $ \sm{G}$ corresponding to $ \sm{G}_{I \cup X, I \cup Y} $ is at least 2. To prove the claim, note that after applying some permutations to its rows and columns of $ \sm{G}_{I \cup X, I \cup Y} $ one can see that $ \sm{G}_{I \cup X, I \cup Y} $ is changed to $ U $, where
$$ U = \left(
\begin{array}{c|c}
T &  \begin{array}{cccc} 
		\pm 1 & \pm 1 & \cdots & \pm 1 \\
		\pm1 & \pm1 & \cdots & \pm1
	\end{array}          \\
	\hline 
	\begin{array}{cccc} 
		\pm 1 & \pm 1 \\
		\pm 1 & \pm 1\\
		\vdots &  \vdots \\
		\pm 1 &  \pm 1
	\end{array} & 	\begin{array}{cccc} 
	0 & \pm 1 & \cdots & \pm 1 \\
	\pm 1 & 0  &  \cdots & \pm 1\\
	\vdots &  \vdots  & \ddots & \vdots\\
	\pm 1 &  \pm 1 & \cdots & 0
	\end{array} 
\end{array}\right),$$
and $ T $ has an odd number of $ -1 $.  After multiplying the rows and columns of $ T $ by $ -1 $ if necessary, $ T $ can be changed to the matrix
		\[\begin{pmatrix}
	1 &  1\\
	1 & -1
	\end{pmatrix}.
	\]
Now, by applying some elementary column operations on $ U $, one can obtain the following matrix:
$$ \left(
\begin{array}{c|c}
\begin{array}{cc}
1 & 1\\
1 & -1
\end{array} &  \begin{array}{cccc} 
0 & 0 & \cdots & 0 \\
0 & 0 & \cdots & 0
\end{array}          \\
\hline 
  \begin{array}{ccc}\\
    \vspace{1em}
 {\Huge \ast}
\end{array} &
              \begin{array}{ccc}
                \\
                \vspace{1em}
{\Huge R }
\end{array} 
\end{array}\right).$$
	Clearly, $ R = I_{k-2} $ and $ \det R = 1 $ modulo 2. This implies that
	\begin{equation*}
	\lvert \det U \rvert  \geq 2, \qquad (\text{over real numbers}).
	\end{equation*}
	We have $ \binom{n-4}{k-2} $ number of subsets $ I $ with the desired property. Note that we have  $ \nop{G} \binom{n-4}{k-2} $ number of $ k \times k $ submatrices which introduced. Because, if $ \sm{G}_{I \cup X, I \cup Y}=\sm{G}_{I^{\prime} \cup X^{\prime}, I^{\prime} \cup Y^{\prime}} $, where $ (X,Y) $ and $ (X^{\prime}, Y^{\prime}) $ are two odd pairs, then
	\begin{equation*}
	I=(I \cup X) \cap (I \cup Y)=(I^{\prime} \cup X^{\prime}) \cap (I^{\prime} \cup Y^{\prime})=I^{\prime},
	\end{equation*}
	and therefore, $ X=X^{\prime} $ and $ Y=Y^{\prime} $.
	 Now, by Theorem \ref{CB}, these submatrices have $ 4\nop{G} \binom{n-4}{k-2} $ contribution in $ S_{k}(A^{2}) $. To complete the proof, note that the $ k \times k $ submatrices of $ \sm{G} $ which considered here and those in the proof of Lemma \ref{llb} (which gives us the first term on the right hand side of \eqref{eq-t}) do not have any intersection.
      \end{proof}
      
	\begin{lemma}\label{lat}
		Let $ G $ be a graph of order $ n $ which is not SC-equivalent  to $ K_{n} $. Then $ G $ has at least one odd pair.
	\end{lemma}
\begin{proof}
	One can see that if $ n \leq 3 $, then $ G $ is SC-equivalent to $ K_{n} $. So, assume that $ n \geq 4 $. Let $ v \in V(G) $ be an arbitrary vertex. By applying a Seidel switching on $ G $ with respect to $ N_{G}[v] $ and $ \overline{N_{G}[v]} $, one can suppose that $ v $ is adjacent to all vertices of $ G $. Let $ v_{i}, v_{j} $ and $ v_{k} $ be vertices of $ G \setminus \{v\} $ such that $ v_{i} $ and $ v_{j} $ are adjacent but $ v_{i} $ and $ v_{k} $ are not adjacent. Then, $ (\lbrace v, v_{i} \rbrace , \lbrace v_{j}, v_{k} \rbrace) $ is an odd pair in $ G $.
	
	If such $ v_{i}, v_{j} $ and $ v_{k} $ do not exist, then one can deduce that $ G \setminus \{v\} $ has no induced subgraph isomorphic to $ K_{1} \cup K_{2} $ or $ \overline{K_{1} \cup K_{2}} $. Therefore, $ G \setminus \lbrace v \rbrace $ is either complete graph or empty graph. Hence $ G=K_{n} $ or using a Seidel switching with respect to $ \lbrace v \rbrace $ and $ G \setminus \lbrace v \rbrace $, we can delete all edges of $ G $ and so $ G=\overline{K_{n}}$.
\end{proof}
\begin{lemma}\label{ltg}
	If $ G $ is a graph of order $ n $ and $ \nop{G} \geq 1 $, then $ \nop{G} \geq 2(n-3)^{2} $.
\end{lemma}
\begin{proof}
	Let $ V(G)= \lbrace v_{1}, \ldots , v_{n} \rbrace $. Without loss of generality, suppose that $ (\lbrace v_{1}, v_{2} \rbrace , \lbrace v_{3}, v_{4} \rbrace) $ is an odd pair. Now, for every $ i, 4 \leq i \leq n $, the parity of the number of edges between $ v_{i} $ and $ \lbrace v_{1}, v_{2} \rbrace $ is different from the parity of the number of edges between $ v_{3} $ and $ \lbrace v_{1}, v_{2} \rbrace $ or $ v_{4} $ and $ \lbrace v_{1}, v_{2} \rbrace $. Hence, exactly one of the $ (\lbrace v_{1}, v_{2} \rbrace , \lbrace v_{3}, v_{i} \rbrace) $ and $ (\lbrace v_{1}, v_{2} \rbrace , \lbrace v_{4}, v_{i} \rbrace) $ is an odd pair. So, if we show the second element of this odd pair by $ V_{i} $ ($ v_{i} \in V_{i} $), then $ (\lbrace v_{1}, v_{2} \rbrace , V_{i}) $ is an odd pair. Similarly, if $ i \geq 4 $ and $ v_{j} \notin V_{i} \cup \{v_{1}\} $ ($ 1 \leq j \leq n $), then exactly one of the $ (\lbrace v_{1}, v_{j} \rbrace , V_{i}) $ and $ (\lbrace v_{2}, v_{j} \rbrace , V_{i}) $ is an odd pair. Denote the first element of this odd pair by $ W_{j} $ and hence, $ (W_{j}, V_{i}) $ is an odd pair. So, we have $ (n-3)^{2} $ odd pairs. Note that for every $ i $ and $ j $, $ (V_{i}, W_{j}) $ is an odd pair, too and is not equal to an odd pair $ (W_{j^{\prime}}, V_{i^{\prime}}) $. Therefore we have at least $ 2(n-3)^{2} $ odd pairs in total, as desired.
	\end{proof}
Now, we state a technical lemma for the cubic polynomials with positive coefficients.
\begin{lemma}\label{ld3}
	If $ f(t)=1+ a t +bt^{2}+ct^{3}  $ is a cubic polynomial with positive coefficients, then
	\begin{equation*}
	C_{\frac{1}{2}} \int_{0}^{\infty} \ln (f(t)) t^{-\frac{3}{2}} dt \geq \sqrt{a+2\sqrt{b+2\sqrt{ac}}}.
	\end{equation*}
\end{lemma}
\begin{proof}
	Since the constant term of $ f $ is 1, one can consider the following factorization
	\begin{equation*}
	f(t)=(1+\alpha_{1} t)(1+\alpha_{2} t)(1+\alpha_{3} t),
	\end{equation*}
	where $ \alpha_{1}, \alpha_{2}, \alpha_{3} \in \mathbb{C} $. 
	Since $ f $ is positive over $ \mathbb{R}^{\geq 0} $, no $ \alpha_{i} $ is a negative real number, because otherwise $ f $ has a positive root, a contradiction. On the other hand, $ f $ has a real root, so at least one of $ \alpha_{i} $, say $ \alpha_{1}$, is a positive real number. Hence, $ \alpha_{2} $ and $ \alpha_{3} $ are either positive real numbers or conjugate complex numbers. Note that by our convention about the arguments of complex numbers, if $ \alpha $ and $ \beta $ are conjugate, then $ \sqrt{\alpha} \sqrt{\beta} = \sqrt{\alpha \beta}$. Now by \eqref{sum}, 
	\begin{equation*}
	\sum\limits_{i=1}^{3}\sqrt{\alpha_{i}}= C_{\frac{1}{2}} \int_{0}^{\infty} \ln (f(t)) t^{-\frac{3}{2}} dt.
	\end{equation*}
	Define $ X=\sum \limits_{i=1}^{3}\sqrt{\alpha_{i}} $. We have
	\begin{align}\label{xx}
	X^{2} & =\sum\limits_{i=1}^{3}\alpha_{i}+2 \sum\limits_{i < j} \sqrt{\alpha_{i} \alpha_{j}} = a +2 \sum\limits_{i < j} \sqrt{\alpha_{i}\alpha_{j}} \\
	\label{eq:1}
	(\sum\limits_{i < j}\sqrt{\alpha_{i}\alpha_{j}})^{2} & = \sum\limits_{i < j}\alpha_{i}\alpha_{j} + 2 \sqrt{\alpha_{1}\alpha_{2}\alpha_{3}} \sum\limits_{i=1}^{3}\sqrt{\alpha_{i}}= b +2(\sqrt{c})X. 
	\end{align}
	In both cases, either $ \alpha_{i} $ is positive or $ \alpha_{2}= \overline{\alpha_{3}}$, the values $ \sum \limits_{i < j} \sqrt{\alpha_{i}\alpha_{j}} $ and 
	$ \sqrt{c} $ are positive numbers. So, Equation \eqref{xx} implies  $ X \geq \sqrt{a}$.
	Now, \eqref{xx} and \eqref{eq:1} imply that,
	\begin{equation*}
	 X= \sqrt{a+ 2 \sum\limits_{i < j} \sqrt{\alpha_{i}}\sqrt{\alpha_{j}}}=\sqrt{a+ 2 \sqrt{b+2 \sqrt{c}X}} \geq \sqrt{a+2\sqrt{b+2\sqrt{ac}}}.
	\end{equation*}
\end{proof}
Now, we are ready to prove Theorem \ref{firstthm2}.
\vskip 2mm
\noindent \textbf{Proof of Theorem 2.}
 As stated in \cite{Ham},  using a computer search one can see that $ \se{G} \geq 2n-2 $ for $ n \leq 10 $. So, we assume that $ n > 10 $. Now, by Lemma \ref{lskA2}, for $ k= 1 , \ldots, n $, we have
\begin{equation}
S_{k}(A^{2}) \geq n(n-1) \binom{n-2}{k-1} + 4 \nop{G} \binom{n-4}{k-2},
\end{equation}
where $ A=\sm{G} $.
Hence, for the eigenvalues $ \lambda_{1}, \ldots , \lambda_{n} $ of $ A $ and $ t > 0 $, we have 
\begin{align*}
\prod \limits_{i=1}^{n} (1+t \lambda_{i}^{2}) & =1+ \sum_{k=1}^nS_k(A^2)t^k\\
 & \geq 1+ \sum\limits_{k=1}^{n}  (n(n-1)\binom{n-2}{k-1}+4\nop{G}\binom{n-4}{k-2})t^{k} \\ 
& = 1+n(n-1)t(1+t)^{n-2}+4\nop{G}t^{2}(1+t)^{n-4} \\ 
&= 1 + (n-4)t(1+t)^{n-2}+(n^{2}-2n+4)t(1+t)^{n-2}+4\nop{G}t^{2}(1+t)^{n-4}\\
&  > (1+t)^{n-4}(1+(n^{2}-2n+4)t(1+t)^{2}+4\nop{G}t^{2}), 
\end{align*}
where the last inequality follows from
\begin{equation*}
(1+t)^{n-4} < 1+ (n-4)t(1+t)^{n-5} < 1+ (n-4)t(1+t)^{n-2}.
\end{equation*}
Therefore we have
\begin{align}\label{en4}
\se{G} & \nonumber = \sum\limits_{i=1}^{n} \lvert \lambda_{i} \rvert =C_{\frac{1}{2}}\int_{0}^{\infty} \ln(\prod \limits_{i=1}^{n} (1+t \lambda_{i}^{2}))t^{-\frac{3}{2}}dt \\ \nonumber
& > C_{\frac{1}{2}}\int_{0}^{\infty} \ln((1+t)^{n-4}(1+(n^{2}-2n+4)t(1+t)^{2}+4\nop{G}t^{2}))t^{-\frac{3}{2}}dt\\ \nonumber
& = C_{\frac{1}{2}}\int_{0}^{\infty} \ln(1+t)^{n-4}t^{-\frac{3}{2}}dt+C_{\frac{1}{2}} \int_{0}^{\infty} \ln(1+(n^{2}-2n+4)t(1+t)^{2}+4\nop{G}t^{2})t^{-\frac{3}{2}}dt\\ 
 & = (n-4) + C_{\frac{1}{2}}\int_{0}^{\infty} \ln(1+(n^{2}-2n+4)t(1+t)^{2}+4\nop{G}t^{2})t^{-\frac{3}{2}}dt.
\end{align}
Let $ g(t)=1+(n^{2}-2n+4)t(1+t)^{2}+4\nop{G}t^{2}=1+\alpha t + (2 \alpha + 4\nop{G})t^{2}+ \alpha t^{3}  $, where $ \alpha=n^{2}-2n+4 $. Hence by Lemma \ref{ld3}, we have
\begin{align*}
C_{\frac{1}{2}}\int_{0}^{\infty} \ln(g(t))t^{-\frac{3}{2}} & \geq \sqrt{\alpha+2 \sqrt{(2\alpha+4\nop{G})+ 2 \alpha}} \\ & = \sqrt{\alpha+4 \sqrt{\alpha+\nop{G}}}.
\end{align*}
Note that for every natural number $ n $, $ \alpha=n^{2}-2n+4 \geq \frac{3}{4}n^{2} $ which implies that
\begin{equation}\label{ex}
C_{\frac{1}{2}}\int_{0}^{\infty} \ln(g(t))t^{-\frac{3}{2}} \geq \sqrt{n^{2}-2n+4+4\sqrt{\frac{3}{4}n^{2}+\nop{G}}}.
\end{equation}
Therefore if $ \nop{G} \geq \frac{3}{2}n^{2} $, we have
\begin{equation*}
C_{\frac{1}{2}}\int_{0}^{\infty} \ln(g(t))t^{-\frac{3}{2}} \geq \sqrt{n^{2}-2n+4+6n} = n+2,
\end{equation*}
which Equation \eqref{en4} yields that, $ \se{G} > 2n-2 $.
Notice that for any graph $ G $, which is SC-equivalent to
$ K_{n} $, $ \se{G}=2n-2 $ holds. Hence we assume that $ G $
is not SC-equivalent to $ K_{n} $. In this case, by Lemma
\ref{ltg}, $ \nop{G} \geq 2(n-3)^{2} $ and so for
$ n \geq 23 $, we have $ \nop{G} \geq \frac{3}{2}n^{2} $, as desired. Also, if $ G $ is not
SC-equivalent to $ K_{n} $, then by Lemma \ref{lat}, there is an odd pair in
$ G $, say
$ (\lbrace v_{1},v_{2} \rbrace , \lbrace v_{3}, v_{4}
\rbrace) $, and as the proof of Lemma \ref{ltg} shows, all
$ 2(n-3)^{2} $ odd pairs of $ G $ contain $ v_{1} $ or
$ v_{2} $. Hence if
$ G \setminus \lbrace v_{1}, v_{2} \rbrace $ is not
SC-equivalent  to $ K_{n-2} $, $ G $ has $ 2(n-5)^{2} $ other odd
pairs. Notice that for every natural number $ n \geq 11 $,
$ 2(n-3)^{2}+2(n-5)^{2} \geq \frac{3}{2}n^{2} $, as desired. To complete the proof, we consider the graph $ G $ with $ 10 < n=\lvert V(G) \rvert < 23 $ and
$ G \setminus \lbrace v_{1}, v_{2} \rbrace $ is SC-equivalent  to
$ K_{n-2} $. Now, by applying the Seidel switching on $ G $ or its complement, one can assume that
$ G \setminus \lbrace v_{1}, v_{2} \rbrace \simeq K_{n-2} $.
The number of these graphs, up to isomorphism, is at most
$ 2n^{2} $ and so by a computer search, it can be easily
checked that the assertion holds for these graphs. The
proof is complete. \QEDB
\vskip 2mm
Now, by \cite[Theorem 13]{obu}, we close the paper with the following corollary.
\begin{corollary}
	For every graph $ G $, $ 	\se{G} \geq \mathcal{E}(G)$.
\end{corollary}


\begin{thebibliography}{9}
  	\bibitem{gho}
  	E. Ghorbani,  On eigenvalues of Seidel matrices and Haemers' conjecture, Designs, Codes and Cryptography, 84. 1--2 (2017) 189--195.
  	\bibitem{gre}
  	G. Greaves, J.H. Koolen, A. Munemasa, F. Sz\"oll\H{o}si, Equiangular lines in Euclidean spaces,  Journal of Combinatorial Theory, Series A, 138, (2016) 208--235.
  	\bibitem{Ham}
  	 W.H. Haemers, Seidel switching and graph energy, MATCH Commun. Math. Comput. Chem. 68 (2012) 653--659.
  	\bibitem{kle}
  	I. Kleme\v s, Finite Toeplitz matrices and sharp Littlewood conjectures,  \selectlanguage{russian} Алгебра и анализ. 13(1) (2001) 39--59.
  	\bibitem{mo}
  	A.W. Marshall, I. Olkin, B.C. Arnold, Inequalities: theory of majorization and its applications, 2nd. edition, Springer Series in Statistics, (2010).
  	\bibitem{obu}
  	M.R. Oboudi, Energy and Seidel energy of graphs, MATCH Commun. Math. Comput. Chem, 75 (2016) 291--303.
  	\bibitem{lint-seidel}
  	J.H. van Lint, J.J. Seidel, Equilateral point sets in elliptic geometry, Indag. Math. 28 (1966) 335--348.
  \end{thebibliography}
\end{document}